\documentclass[letterpaper, 10pt, conference]{ieeeconf}
\usepackage{cite}
\usepackage{xcolor}
\IEEEoverridecommandlockouts                              
\let\labelindent\relax
\usepackage{enumerate}
\usepackage{amsthm,amssymb,verbatim, mathtools,enumerate,algorithm}
\usepackage{cite}
\usepackage{diagbox}
\usepackage{times,algpseudocode,tikz}
\usepackage{booktabs}
\usepackage{subcaption}

\allowdisplaybreaks[2]

\newcommand{\reals}{{\mathbb{R}}}

\newcommand{\im}{\mathop{\bf im}} 
\renewcommand{\ker}{\mathop{\bf ker}} 
\newcommand{\vvec}{\mathop{\bf vec}}

\newcommand{\Tr}{\mathop{\bf Tr}}
\newcommand{\diag}{\mathop{\bf diag}}

\newcommand{\argmin}{\mathop{\rm argmin}}

\newcommand{\norm}[1]{\left\lVert#1\right\rVert}
\newcommand{\mnorm}[1]{{\left\vert\kern-0.25ex\left\vert\kern-0.25ex\left\vert #1 
    \right\vert\kern-0.25ex\right\vert\kern-0.25ex\right\vert}}

\newcommand{\mc}{\mathcal}

\newtheorem{theorem}{Theorem}
\newtheorem{lemma}{Lemma}
\newtheorem{corollary}{Corollary}

\newtheorem{proposition}{Proposition}

\newtheorem{example}{Example}

\newcommand{\tc}[2]{\textcolor{#1}{#2}}
\newcommand{\reviewi}[1]{\tc{black}{#1}}
\newcommand{\reviewii}[1]{\tc{black}{#1}}
\newcommand{\reviewiii}[1]{\tc{black}{#1}}

\usepackage{hyperref}
\usepackage[dvipsnames]{xcolor}
\usepackage[normalem]{ulem}

\hypersetup{
    colorlinks=true,
    linkcolor=blue,
    filecolor=magenta,      
    urlcolor=cyan,
}

\newcommand{\rre}{\text{rre}}

\begin{document}
\title{\Large \bf On the Gap Between $H_2$
Optimal Control and Disturbance Decoupling
}
\author{
Ruirui Ma and
Sarah H.Q. Li
\thanks{Ruirui Ma and Sarah~H.~Q.~Li are with Georgia Tech Institute for Robotics and Intelligent Machines (IRIM), Atlanta, GA, USA (emails \{rma75, sarahli\}@ gatech.edu) 
}%
}
\maketitle
\begin{abstract}
We study the relationship between disturbance decoupling (DD) and $H_2$ optimal control for linear time-invariant (LTI) systems, revealing a fundamental gap between DD subspace constraints and semi-definite program (SDP)-based $H_2$ minimization. We show that DD is equivalent to the existence of zero $H_2$ gain without requiring internal stability, whereas SDP-based $H_2$ minimization strictly optimizes over stabilizing controllers and therefore fails to recover DD controllers when the closed-loop dynamics may be marginally stable. Moreover, we show that the trace representation of $H_2$ norms further biases solutions away from complete DD. Motivated by this, we formulate a bilinear matrix inequality (BMI)-constrained optimization program that directly enforces the DD subspace condition to compute DD controllers. We propose a difference-of-convex (DC) iterative algorithm that preserves DD and stability at every iteration, and establish its convergence to Karush--Kuhn--Tucker (KKT) points under standard constraint qualification conditions. Numerical experiments on a four bus power network demonstrate that the proposed algorithm achieves significantly better disturbance rejection while enabling optimization of additional performance metrics. The resulting framework establishes a computationally tractable link between geometric DD theory and optimization-based controller design.

\end{abstract}
\section{Introduction}
Feedback control is often deployed in environments where disturbances propagate across interconnections~\cite{satouri2021disturbance,cerda2012control,lunze2006control}. In multi-agent coordination~\cite{li2020disturbance,altafini2012consensus}, sensor networks~\cite{lebon2025optimal}, and power networks~\cite{Yi2019,Yi2024}, disturbances spread through linkages and degrade global performance. In such settings, it is often not sufficient to attenuate disturbances in an aggregate sense; \reviewiii{instead, one may require that specific performance outputs be decoupled from subsystem disturbances~\cite{wonham1979linear}}.

Geometric control theory provides a complete geometric characterization of DD through invariant subspaces\reviewii{~\cite{lebon2025optimal,lebon2026,basile1992controlled,trentelman2001control}}\reviewiii{,~\cite{marro2002,prattichizzo2008}}. \reviewiii{Graph theoretic methods can characterize the conditions for generic solvability of DD for structured systems~\cite{vanderwoude1994, vanderwoude1996}}. 
While elegant and complete from a structural perspective, existing theory focus on feasibility and do not directly yield optimization-compatible synthesis procedures. As a result, combining DD with additional objectives such as tracking error, convergence rate is challenging within real-time optimal control~\cite{garcia1989model}. In contrast, modern control design for \reviewi{disturbance rejection} is dominated by convex optimization approaches\reviewii{~\cite{eren2017model,Yi2024}}, particularly $H_2, H_\infty$ control\reviewii{~\cite{boyd1994linear,wu2025,schuchert2024}}, which leverages \reviewi{convex programs} to achieve \reviewi{disturbance attenuation} and internal stability. \reviewi{Existing literature on applying optimization to geometric control for disturbance decoupling is sparse~\cite{lebon2025optimal, lebon2026}.} 
This paper establishes a precise connection between these two paradigms \reviewii{ by combining geometric control conditions with optimization-based control synthesis}. We show that DD is equivalent to achieving zero $H_2$ gain from disturbance to output, but this equivalence holds without requiring internal stability. Additionally, we demonstrate how geometric subspace conditions are \emph{better} at enforcing DD than $H_2$ norm minimization within optimization frameworks.  

\textbf{Contributions}: \reviewii{We develop a framework that bridges optimization and $H_2$ control approaches to DD.}
First, we show that DD controllers achieve zero $H_2$ gain from disturbance to output under necessary and sufficient conditions without requiring internal stability. 
Second, we demonstrate that when DD with internal stability is achievable, SDP-based optimization approaches to $H_2$ minimization still do not achieve complete DD due to numerical errors.  
Finally, we use subspace conditions to directly constrain optimal control formulations to form a bilinear matrix inequality (BMI)-constrained optimal control problem that satisfies standard constraint qualifications. We develop a difference-of-convex programming algorithm that solves this problem while guaranteeing internal stability and DD at every iteration. On a power network example, we demonstrate that this algorithm achieves exponentially smaller $H_2$ gains with exponentially less control effort required. 

\section{Literature Review}
DD  originates from geometric control, where feasibility conditions are characterized via controlled-invariant geometric subspaces~\cite{wonham1979linear,basile1992controlled}. These results provide  constructive procedures for state-feedback DD with and without internal stability as invariant subspace procedures rather than optimization programs~\cite{trentelman2001control,Sarsilmaz2024}. Extensions to measurement feedback, structured systems, and stochastic hybrid dynamics retain the geometric flavor and focus primarily on feasibility and structural characterization~\cite{Stoorvogel1992,Monshizadeh2015}. In parallel, modern controller synthesis has evolved toward Lyapunov certificate-constrained optimization frameworks via robust control techniques such as $H_2$ and $H_{\infty}$ control~\cite{zhou1998essentials,stoorvogel1992singular,marino1994nonlinear,feron1992numerical}, which are \emph{almost disturbance decoupling}.  Consequently, there remains a need for optimization-compatible formulations of DD that codesigns stability, convergence rate, and dynamical structure within modern controller synthesis pipelines.
\section{Preliminaries}
\textbf{Notation}. The image and the kernel of matrix $Y \in \reals^{n \times m}$ are denoted by $\im{Y}$ and $\ker{Y}$, respectively.
The Kronecker product of $Y$ and $Y_2 \in \reals^{p \times q}$  is denoted by $Y \otimes Y_2$. 
The $k \times k$ identity matrix is denoted by $I_k$, and  the $i^{th}$ canonical basis vector in $\reals^n$ is $e_i$.
\reviewiii{Optimal solutions are indicated by superscript $(\cdot)^\star$. Reduced row echelon is denoted by $rre$.}

Consider a linear time-invariant (LTI) system~\cite[Chp.  4]{trentelman2001control}, given by 
\begin{equation}\label{eqn:lti}
\begin{aligned}
       \dot x(t) &= Ax(t) + Bu(t) + Ed(t), \\
        z(t) &= Hx(t),  \quad \forall t \in \reals_+,
\end{aligned}
\end{equation}
where $x(t) \in \reals^{n}$ is the state, $u(t) \in \reals^{m}$ is the control input, $z(t) \in \reals^p$ is the output, and $d(t) \in \reals^{\ell}$ is the disturbance. We consider all disturbance functions that are locally integrable functions in time, given by 
\begin{equation}\label{eqn:dd_noise_set}
   \textstyle\mc{D} = \{d:\reals_+\mapsto \reals^\ell  \mid \int_0^T \norm{d(\tau)} d\tau < \infty, \ \forall \ T \in \reals_+ \}. 
\end{equation} 
Set $\mc{D}$ allows disturbances with unbounded total magnitude over the infinite time horizon, while still guaranteeing well-posed system trajectories.
To mitigate effects of $\mc{D}$ on~\eqref{eqn:lti}, we strictly consider linear state-based feedback, $u(t) = Fx(t) \in \reals^m, \ \forall \ t \in \reals_+$, where $F \in \reals^{m\times n}$ is the feedback controller.
The closed-loop dynamics of~\eqref{eqn:lti} under controller $F$ and without any disturbance ($d(t) = 0, \, \forall \, t \in \reals_+$) is given by
\begin{equation}\label{eqn:lti_df}
\dot x_{dd}(t) = A_Fx_{dd}(t), \,\, A_F = A+BF, \,\, z_{dd}(t) = Hx_{dd}(t).
\end{equation}
We say that $F$ is a \textbf{DD controller} if the disturbance-free output $z_{dd}(t)$~\eqref{eqn:lti_df} and the disturbed output $z(t)$~\eqref{eqn:lti}  satisfy 
\begin{equation}\label{eqn:dd_condition}
   \textstyle e(t) = z_{dd}(t) - z(t) = 0, \, \forall \, x_{dd}(0) = x(0), \, \forall \ t \in \reals_+, d \in \mc{D}.
\end{equation}
DD ensures that the transient behavior of the closed-loop LTI system $z_{dd}(t)$ is preserved for all disturbances. In certain applications such as power networks~\cite{Yi2019}, multi-agent learning~\cite{li2020disturbance}, and opinion-dynamic networks~\cite{lebon2025optimal},  preserving smooth transient system response may be more desirable than  complete disturbance rejection. 

In addition to DD control,  we consider controllers that simultaneously  ensure asymptotic stability.  We say that~\eqref{eqn:lti_df}  is \textbf{internally stable} if  $A_F$ is Hurwitz~\cite[Def. 4.31]{trentelman2001control}, i.e., 
\begin{equation}\label{eqn:internal_stab}
    \lim_{t \to \infty} x_{dd}(t) = 0, \quad \forall \, x_{dd}(0) \in \reals^n.
\end{equation}
\reviewiii{A controller $F$ is internally stabilizing and disturbance decoupling \textbf{(DDPF)} if it satisfies both~\eqref{eqn:dd_condition} and~\eqref{eqn:internal_stab}.}

System~\eqref{eqn:lti} has a {DD controller} if and only if   $(A,B,E,H)$ satisfy a geometric subspace  condition given in~\cite[Thm. 4.8]{trentelman2001control}. 
We provide a matrix-based formulation to enable optimization-based approach for controller synthesis later.

\begin{proposition}[\cite{Sarsilmaz2024,trentelman2001control}]\label{prop:la_dd_condition}
A state feedback controller $F\in \reals^{m\times n}$ is a {DD controller} for~\eqref{eqn:lti} if and only if there exist matrices $X\in \reals^{k \times k}$, $V\in \reals^{n \times k}$ such that
\begin{equation}\label{eqn:dd_necessary_sufficient}
\begin{aligned}
 VX - BFV = AV,  \quad \im{E} \subseteq \im{V} \subseteq \ker{H}.
\end{aligned}
\end{equation}
\end{proposition}



In Proposition~\ref{prop:la_dd_condition}, $\im(V)$ is a control invariant set of~\eqref{eqn:lti_df}. Different control invariant sets correspond to different sets of DD controllers; some corresponds to larger sets of DD controllers than others. We refer to~\cite{Sarsilmaz2024} for detailed discussions on control invariant sets as well as algorithms to compute them. Feasibility of DDPF controllers is also well-characterized and provided in~\cite{Sarsilmaz2024}.



\textbf{$\boldsymbol{H_2}$ Norm Minimization}.
An alternative approach for disturbance rejection is via SDP-based $H_2$ norm  minimization~\cite{zhou1998essentials}. When $u(t) = Fx(t)$, the impulse response from disturbance $d(t)$ to  output $z(t)$ in~\eqref{eqn:lti} is given by $g(t)=He^{A_Ft}E$. The corresponding $H_2$ norm~\cite[Sec.4.3]{zhou1998essentials} is given by
\begin{equation}\label{eqn:h2_norm}
\textstyle\norm{g}_{H_2}^2 = \int_{0}^\infty \Tr(g(\tau)^\top g(\tau))d\tau. 
\end{equation}
When $A_F$ is Hurwitz, the $H_2$ norm~\eqref{eqn:h2_norm} is equivalent to $\Tr(E^\top W_o E)$, where $W_o$ is the observability Gramian given by the Lyapunov equation $A_F^\top W_o+W_oA_F^\top=-H^\top H$~\cite[Lem.4.4]{zhou1998essentials}. Additionally,  if $d(t)$ has finite $L_2$ norm, such that  $\int_{0}^\infty \norm{d(\tau)}^2 d\tau <  \infty$, then \reviewiii{the output} $L_2$ norm is \reviewiii{also finite,} $\int_0^\infty \norm{z(\tau)}_2^2 d\tau  < \infty$.
Over stabilizing controllers, $H_2$ norm minimization can be formulated as a SDP, given by
\begin{equation}\label{eqn:h2_LMI}
    \begin{aligned}
    \underset{G, N, P, W}{\min}\quad &  \text{Tr}(W)\\
    \text{s.t. } \quad & 
    \begin{bmatrix}
        G & PH^T & P\\
        HP & I_p & 0\\
        P  & 0 & \epsilon^{-1} I_n
    \end{bmatrix}\succeq 0\\
    & G=-(BN+AP)^\top -(BN+AP)\\
    & \begin{bmatrix}
        W & E^\top\\
        E & P
    \end{bmatrix}\succeq 0,  P\succ 0, \reviewiii{P \in \reals^{n\times n},}
    \end{aligned}
\end{equation}
where the optimal controller $F^\star = N^\star (P^\star)^{-1}$ achieves the optimal $H_2$ norm $\norm{g}_{{H}_2}^2 = \text{Tr}(E^\top W_o E)$~\cite{feron1992numerical}. In the SDP formulation, $H_2$ minimization implicitly assumes asymptotic stability of the closed-loop system $A_F$. 

\section{Bridging $H_2$ minimization and DD}
%
Although the relationship between $H_2$ norm minimization and DD has been known since~\cite{trentelman2001control}, the theoretical and computational gaps between the two remain poorly understood, particularly when computing $H_2$ minimal controllers via the SDP formulation~\eqref{eqn:h2_LMI}. In this section, we compare DD controllers without internal stability and $H_2$-norm-minimizing controllers and reveal a clear gap between the two. 
\subsection{Zero $H_2$ norm and DD equivalence}
Through its usage of Gramians, SDP-based $H_2$ minimization is performed implicitly over controllers that are internally stable.
In fact, having a finite norm does not  strictly require internal stability, and the scenarios where this is true is exactly characterized by DD control. 
We first show that in addition to being necessary and sufficient for the existence of a DD controller, Proposition~\ref{prop:la_dd_condition} is also necessary and sufficient for achieving a zero $H_2$ norm. 
\begin{lemma}\label{lem:dd_implies_h2}
Controller $F \in\reals^{m\times n}$ achieves zero $H_2$ norm~\eqref{eqn:h2_norm} if and only if there exist matrices $X\in \reals^{k \times k}$, $V\in \reals^{n \times k}$ such that~\eqref{eqn:dd_necessary_sufficient} holds. 
\end{lemma}
\begin{proof}
    To see that DD implies zero $H_2$ norm, if $F$ is a DD controller, $g(t)\equiv0$\cite{trentelman2001control}. Hence, the $H_2$ norm $\norm{g}_{{H}_2}^2=\int_0^\infty \Tr(g(t)^\top g(t))dt=0$. Conversely, if $\norm{g}_{{H}_2}^2 = 0$, then $g$ is zero almost everywhere. In LTI, $g(t)=He^{A_Ft}E$ is a continuous function. Therefore, $g(t)\equiv0$ and $F$ is DD. 
\end{proof}
Lemma~\ref{lem:dd_implies_h2} shows that $H_2$ norm is zero if and only if the controller is DD. Critically, Lemma~\ref{lem:dd_implies_h2} does not require $F$ to be internally stable. \reviewiii{This implies that minimizing $H_2$ via solving~\eqref{eqn:dd_necessary_sufficient} will include  marginally stable controllers, and thus achieve a lower $H_2$} gain than minimizing $H_2$ gain via SDP~\eqref{eqn:h2_LMI}, which strictly optimizes over internally stable controllers. We show below that the difference can be nontrivial.
\begin{example}[Disjoint DD and SDP controllers]\label{ex:1}
In system~\eqref{eqn:lti}, let 
$A = \begin{bmatrix}
    a_1 & a_2 & a_3
\end{bmatrix}$, $a_1^\top = [0 \,\, -1 \,\, 1]$, $a_2^\top = [1 \,\, -1 \,\, 0]$, $a_3^\top= [0 \,\, 0 \,\, 1]$,
    $B=e_2$, $E=e_3$, $H=e_1^\top$. Its control invariant DD subspace is $V = e_3$. From~\eqref{eqn:dd_necessary_sufficient}, we find a DD controller $F_{dd} = \begin{bmatrix}
        0&0&0
    \end{bmatrix}^\top$ with zero $H_2$ norm~\eqref{eqn:h2_norm}.
    Solving~\eqref{eqn:h2_LMI}, we find a controller $F_{H_2}=\begin{bmatrix}
        -8.1e5 & -8.0e1 & -1.6e6
    \end{bmatrix}^\top$ whose $H_2$ norm is $2.0$. The closed-loop dynamics of $A+BF_{dd}$ is marginally stable. 
\end{example}


\subsection{SDP and DD equivalence under internal stability}
Next, we demonstrate that over the set of finite and \emph{internally stabilizing} controllers, the SDP formulation~\eqref{eqn:h2_LMI} and the DD condition~\eqref{eqn:dd_necessary_sufficient} have the same solution set. 
\begin{proposition}\label{prop:h2_dd_condition}
The SDP~\eqref{eqn:h2_LMI} has a zero $H_2$ norm at optimal solution $N^\star, P^\star, W^\star$ if and only if $F^\star = N^\star (P^\star)^{-1}$ satisfies~\eqref{eqn:dd_necessary_sufficient} and $(A+BF^\star)$ is Hurwitz.
\end{proposition}
\begin{proof}
The forward implication follows from Lemma~\ref{lem:dd_implies_h2}. \reviewiii{If a controller $F$ satisfies (6), it is DD. If $A_F$ is also internally stable, then $F$ solves SDP (8) and the closed loop system's} $H_2$ norm is $\text{Tr}(E^\top W_o E)$ , where $W_o$  is the observability Gramian. Since $F, W_o$ is feasible for~\eqref{eqn:h2_LMI} and achieves zero objective (lowest possible for $H_2$ objective), they are optimal for~\eqref{eqn:h2_LMI}. 
\end{proof}



Proposition~\ref{prop:h2_dd_condition} suggests that when DDPF controllers exist, the solution set to SDP~\eqref{eqn:h2_LMI} and the subspace condition~\eqref{eqn:dd_necessary_sufficient} with $A_F$ Hurwitz constraints should be equivalent. However, $H_2$ norm does not distinguish between impulse responses $g(t)  = 0$ everywhere and $g(t) = 0$ \emph{almost everywhere}. Consequently,  impulse responses that are sharply concentrated near $t=0$ can achieve arbitrarily small $H_2$ norms. In practice, optimization solvers under numerical inaccuracy cannot distinguish between $\epsilon$ $H_2$ norms, creating bias away from true disturbance decoupling controllers characterized by~\eqref{eqn:dd_necessary_sufficient} and towards controllers whose impulse responses  with $\epsilon$ $H_2$ norms, is a commonly observed solution to SDP~\eqref{eqn:h2_LMI}. On the other hand, \eqref{eqn:dd_necessary_sufficient} is a subspace condition and does not have this numerical issue. This issue creates a \emph{computation gap} between DDPF solutions and SDP--based $H_2$ optimal control that also persists for other forms of robust optimal controllers that directly optimize over norms (e.g. $H_\infty$). This gap creates observable performance changes under larger disturbance magnitudes. 
\begin{example}[Computational Gap between DD and SDP]\label{ex:computational_gap}
    Consider the LTI system in Example~\ref{ex:1} with the third column of $A$ replced by $\begin{bmatrix}
        0&1&-1
    \end{bmatrix}^\top$. 
    For this system, solving~\eqref{eqn:h2_LMI} gives a controller $F_{H_2}=[-4.1e5 \, -1.0e2 \, -8.4e-1]^\top$ that achieves an optimal $H_2$ norm of $8.1e-14$. However, no matrices $X,V$ exist such that $VX-BF_{H_2}V=AV$. Meanwhile, a DD controller $F_{dd} = \begin{bmatrix}
        0&0&-1
    \end{bmatrix}^\top$ achieves zero $H_2$ norm. 
    In Figure~\ref{fig:dd_no_dd_compare}, the cumulative error $e_{cum}(T) =\int_{t=0}^T\norm{z(t)-z_{dd}(t)}_2dt$ at $T=10$
    between the two controllers show clear numerical error issues in $F_{H_2}$. 
    \begin{figure}[h!]
    \vspace{-8pt}
    \centering
    \includegraphics[width=1\columnwidth]{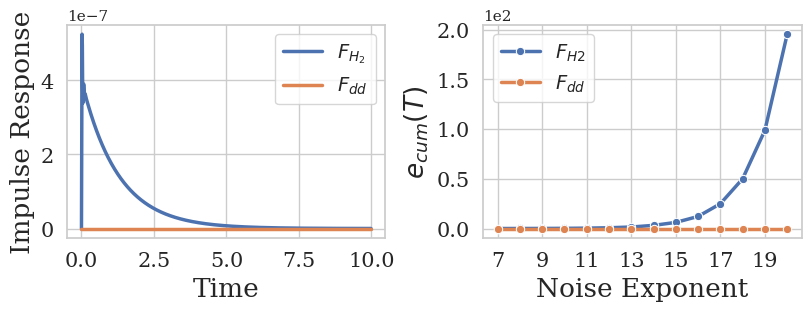}
    \caption{The impulse response and cumulative output error $e_{cum}(10)$. Noise magnitude $d(t)\sim\mc{N}(0,2^l),7\leq \ell \leq 20$.}
    \label{fig:dd_no_dd_compare}
    \end{figure}
\end{example}

\section{Optimization-based Synthesis of DDPF}
As shown in Example~\ref{ex:computational_gap}, $H_2$ norm minimization via SDP~\eqref{eqn:h2_LMI} induces computational errors. Even when DDPF controllers exist, SDP~\eqref{eqn:h2_LMI} may not recover them (in fact, we observed that~\eqref{eqn:h2_LMI}  rarely recovers DDPF). 
Motivated by this, we develop an optimal control formulation that directly enforces~\eqref{eqn:dd_necessary_sufficient}. This also enables DD to be imposed in conjunction with additional metrics (e.g. convergence rate, actuation effort) over internally stabilizing controllers. 


We consider an optimal control problem in which internal stability and DD as the subspace condition~\eqref{eqn:dd_necessary_sufficient} are enforced as constraints. In formulation given below, we assume that the matrix $V$ is pre-computed via existing algorithms~\cite{Sarsilmaz2024}. 
\begin{subequations}\label{eqn:unified_framework}
    \begin{align}
    \underset{\alpha, F, P,X}{\min} & J(\alpha, F, P, X)\\
    \text{s.t. } & A_F^\top P + PA_F + \Pi(P, \alpha)\preceq 0,\label{subeqn:unified_lyap} \\
    &  V X - BFV = AV,\label{subeqn:unified_dd}\\ 
    &   P \succ 0^{n\times n}, X \in \reals^{k\times k},\  F \in \reals^{m\times n},\  \alpha \in \reals.
    \end{align}
\end{subequations}

\textbf{Modeling Capability}. The optimization program in~\eqref{eqn:unified_framework} provides a unified modeling framework that combines the DD subspace condition~\eqref{subeqn:unified_dd} with standard closed-loop performance objectives~\eqref{subeqn:unified_lyap}, where the term $\Pi(P,\alpha)$ encodes the desired \reviewi{stability} certificate. \reviewi{This enables the selection of a controller that optimizes the specified performance objective within the family of DD controllers}. Several important control objectives are captured by this formulation, 
\begin{enumerate}
    \item Internal stability: $\Pi(P, \alpha) = H^\top H + \epsilon I$. 
    \item Convergence: $\Pi(P, \alpha) = 2\alpha P$ (convergence rate is $2\alpha$). 
    \item Without internal stability: $\Pi(P, \alpha) = -(A_F^\top P + PA_F)$ (removing constraint~\eqref{subeqn:unified_lyap}).
\end{enumerate}
\reviewi{For the choice of $\Pi(P,\alpha)$ without internal stability, program~\eqref{eqn:unified_framework} optimizes over the set of all DD controllers, including those that are unstable and marginally stable.} 


\textbf{Problem Nonconvexity.}
The program~\eqref{eqn:unified_framework} is a BMI-constrained optimization problem for performance objectives (1) and (2) listed above. Without the DD condition~\eqref{subeqn:unified_dd}, it is identical to the controller design problem for state-output feedback~\cite{feron1992numerical,boyd1994linear} and as such, can be convexified via Schur complement. However, the convexification of~\eqref{subeqn:unified_lyap} transforms~\eqref{subeqn:unified_dd} into a \emph{BMI}, which cannot be convexified further. To illustrate this, let $\Pi(P, \alpha) = 0$, $Q\coloneqq P^{-1}$, and $N=QF$, then~\eqref{subeqn:unified_lyap} is given by
\begin{equation}
    \begin{aligned}
        (QA+BN)^\top+(AQ+BN)&\preceq 0\\
        V^\star X -AV^\star - BNQ^{-1}V^\star &=0
    \end{aligned}
\end{equation}
which produces a bilinear term $NQ^{-1}$ that cannot be further convexified.

\subsection{Local optimality via constraint qualification.}
Due to its bilinear coupling between decision variables, ~\eqref{eqn:unified_framework} cannot be solved to global optimality in general. Instead, we focus on locally optimal solutions by computing~\eqref{eqn:unified_framework}'s KKT points. To ensure that KKT conditions  correspond to locally optimal solutions and therefore~\eqref{eqn:unified_framework} are suitable for primal--dual optimization methods, we first verify that the Robinson's constraint qualification (RCQ) holds~\cite[Chp.3]{bonnans2013perturbation}.

\begin{theorem}[Robinson's CQ Satisfaction]\label{thm:rcq}
All locally optimal solutions of~\eqref{eqn:unified_framework} satisfy its KKT conditions if both of the following conditions hold:
\begin{enumerate}
    \item  there exists $P \succ\! 0$ such that $A_F^\top P + PA_F + \Pi(P, \alpha)\prec\!0$;
    \item  $V $ has full column rank, and $\im{V} + \im{B} = \reals^n$.
\end{enumerate}
\end{theorem}
\begin{proof}
We show that the KKT conditions are necessary for local optimality by showing that~\eqref{eqn:unified_framework} satisfies Robinson's Constraint Qualification (RCQ)~\cite[Eqn.3.13]{bonnans2013perturbation}. \reviewiii{Let $\mathcal{X}\coloneqq\mathbb{R}\times \mathbb{R}^{m\times n}\times \mathcal{S}^n\times \mathbb{R}^{k\times k}$, $z\coloneqq(\alpha,F,P,X)\in \mathcal{X}$, $\mathcal{Y}_1\coloneqq \mathcal{S}^n\times\mathcal{S}^n$, and $\mathcal{Y}_2\coloneqq \mathbb{R}^{n\times k}$, }
the feedback optimal control problem~\eqref{eqn:unified_framework} can be equivalently expressed as
\begin{equation*}
\textstyle\min_{z\in\mc{X}}\  J(z)  \text{    s.t. } \mathcal E(z) = 0,  \textstyle M(z) \in \mathcal{S}^{n}_+ \times \mathcal{S}^{n}_+,
\end{equation*}
where $M:\mathcal{X}\rightarrow\mathcal{Y}_1$, $\mathcal{E}:\mathcal{X}\rightarrow\mathcal{Y}_2$ are defined as $M(z) = \begin{bmatrix}
    P \\ -(A_F^\top P + PA_F + \Pi(P, \alpha))
\end{bmatrix} $ and $\mathcal E(z) = VX -BFV-AV$. \reviewi{Since 1) $\mathcal{X}$ and $\mathcal{Y}_1\times\mathcal{Y}_2$ are Banach spaces, 2) $(M,\mathcal{E}):\mathcal{X}\rightarrow\mathcal{Y}_1\times\mathcal{Y}_2$ is continuously differentiable, 3) $\mathcal{S}_+^n\times \mathcal{S}_+^n\times\{0\}\subset \mathcal{Y}_1\times\mathcal{Y}_2$ is a closed convex set, and 4) $\mathcal{S}_+^n\times \mathcal{S}_+^n$ has a nonempty interior, then an equivalent condition to RCQ~\cite[Cor.2.101] {bonnans2013perturbation} at $z$ is i) $D\mathcal{E}(z)$ is surjective and
\begin{equation*}
    ii)\, \exists \, z_0 \in\text{Ker}DG(z), M(z) + DM(z)z_0 \in \textbf{Int}(\mathcal{S}^{n}_+ \times \mathcal{S}^{n}_+),
\end{equation*}
}
where $D(\cdot)$ is the gradient with respect to $z$. 
We first note that since $V\in\reals^{n\times k}$ has full column rank, there exists $L \in \reals^{k\times n}$ such that $LV = I_{k\times k}$. Let any $N\in\reals^{m\times k}$, then $F = NL$ ensures $FV=N$. Finally, since $X, N$ are unconstrained and $\im{V}+\im{B}=\mathbb{R}^n$, then linear map $\mathcal{E}(z)$ is surjective, which implies surjectivity of $D\mc{E}(z)$, satisfying condition i). Next, we select $z_0=0$, then condition ii) is equivalent to $M(z) \in (\mathcal{S}^{n}_{++} \times \mathcal{S}^{n}_{++}),$
implying a positive definite $P$ that strictly satisfies~\eqref{subeqn:unified_lyap}, which is also stated as the first condition in \reviewiii{Theorem~\eqref{thm:rcq}}.
\end{proof}
\begin{corollary}
    Condition (2) from Theorem~\ref{thm:rcq} \reviewiii{can always be satisfied by replacing}~\eqref{subeqn:unified_dd}  with its \reviewiii{rre form}, given by 
    \begin{eqnarray}\label{eqn:lme_dd_rre}
     \underbrace{\rre\Bigg(\begin{bmatrix}
    I_{k} \otimes V &  V^\top  \otimes (-B) 
    \end{bmatrix}\Bigg.}_\text{$C$}  \underbrace{\begin{bmatrix}
    		\vvec X \\ \vvec F
    \end{bmatrix}}_\text{$y$} \Bigg|\Bigg. \underbrace{\Bigg.\vvec(AV)\Bigg)}_\text{${b}$}.
    \end{eqnarray}
\end{corollary}
\begin{proof}
    Constraint~\eqref{subeqn:unified_dd} can be written as the vectorized equation $Cy=b$. Its rre form~\eqref{eqn:lme_dd_rre} defines an equivalent constraint whose associated linear map is surjective. Hence, RCQ equivalent condition i) stated in the proof of Theorem~\ref{thm:rcq} is satisfied.
\end{proof}

\subsection{Solving for DDPF}
We showed that~\eqref{eqn:unified_framework} has well-conditioned KKT points. However, computing feasible KKT points remains nontrivial, given~\eqref{eqn:unified_framework} is nonconvex and incompatible with existing BMI solvers such as PENBMI~\cite{henrion2005solving}. We present an inner convex approximation scheme that computes the KKT points  of~\eqref{eqn:unified_framework} while preserving DD and stability at each iteration.  

\textbf{Inner convex approximation method}. We consider the following inner approximation technique from~\cite{dinh2011combining}. 
\begin{lemma}[\reviewi{\cite{dinh2011combining}}]\label{lem:inner_approx}
    For matrix $Z \in \reals^{m\times n}, Y \in \reals^{n\times m}$, the matrix $Y^\top Z + Z^\top Y$ is inner approximated at $(Z^k,Y^k)$ by 
    \begin{equation}\label{eqn:v_1}
        Y^\top Z + Z^\top Y \!\preceq\! G_1(Z, Y) -\reviewi{G_2(Z^k,Y^k)} - F(\Delta Z, \Delta Y, Z^k, Y^k),
    \end{equation}
    where $G_1(Z,Y) = (Z+Y)^\top(Y+Z)$ , $G_2(Z,Y) = Z^\top Z + Y^\top Y$ and 
    \[F = \Delta Z^\top Z^k + (Z^k)^\top \Delta Z + \Delta Y^\top Y^k + (Y^k)^\top \Delta Y.\]
\end{lemma}

Since $Y^\top Z + Z^\top Y = G_1(Z, Y) - G_2(Z,Y)$, ~\eqref{eqn:v_1} is the best convex approximation achievable, which can be transformed into LMIs via Schur complements.
Applying Lemma~\ref{lem:inner_approx} to the BMI constraint~\eqref{subeqn:unified_lyap}, we formulate an inner approximation to~\eqref{eqn:unified_framework} around a feasible point \reviewi{$z^k=(\alpha^k, F^k, P^k, X^k)$} \reviewi{ by linearizing the constraints and augmenting the objective with a penalty around $z^k$}. 
\begin{equation}\label{eqn:unified_framework_linearized}
    \begin{cases}
    \underset{\alpha, P, F,X}{\min} & J(\alpha, F, P, X)+ \frac{\gamma}{2} \sum_{M \in \{\alpha,P,F\}} \norm{M - M^k}_2^2\\
    \text{s.t. } 
    & \begin{bmatrix}
        \reviewi{G_2(Z^k,P^k)+DG_2^k} & (Z+P)^\top\\
        (Z+P) & I
    \end{bmatrix}\succeq 0\\
     &  V X - BFV = AV,\\
     &0_{n\times n}\prec P, \ F \in \reals^{m\times n}, \ \alpha \in \reals, \ X \in \reals^{k\times k}.
    \end{cases}
\end{equation}
When $\Pi(P,\alpha)=2\alpha P$, $Z$ and $DG^k_2$ \reviewi{(the gradient of $G_2$ at $Z^k, Y^k$)}  are given by
\begin{align*}
    DG_2^k&=(B\Delta F + \Delta \alpha I)^\top Z^k+(Z^k)^\top(B\Delta F + \Delta \alpha I)+\\
    &\quad\ \Delta P^\top P^k+(P^k)^\top \Delta P, \quad Z=A+BF+\alpha I,
\end{align*}
and when $\Pi(P,\alpha)=H^\top H + \epsilon I$, $Z$ and $DG^k_2$ are given by
\begin{align*}
    DG_2^k&=(B\Delta F)^\top A_F^k+(A_F^k)^\top(B\Delta F) -H^\top H-\epsilon I + \\
    &\quad\ \Delta P^\top P^k+(P^k)^\top \Delta P, \quad Z=A+BF, 
\end{align*}
where $Z^k=A+BF^k+\alpha^k I$, and $A_F^k = A+BF^k$. Without loss of generality, we assume $\Pi(P,\alpha)=2\alpha P$ for the following discussions.


\begin{algorithm}[ht!]
\caption{Successive linearization of BMI}
\begin{algorithmic}[1]
\Require \(f\), \(\alpha_0\),  \(P_0\), \(F_0\) , \(\epsilon\).
\Ensure \(\alpha^\star, F^\star, P^\star, X^\star\).
\For{\(k=0,1, \ldots\)}
    \State{Solve~\eqref{eqn:unified_framework_linearized} for $(\alpha^{k+1}, P^{k+1},  F^{k+1}, X^{k+1})$} 
    \If{\( \sum_{M \in \{\alpha, P, F\}} \norm{M^{k+1} - M^k}^2_2 > \epsilon\)}
    \State{\((\alpha^{k}, P^{k},  F^{k}, X^{k})=(\alpha^{k+1}, P^{k+1},  F^{k+1}, X^{k+1})\)}
    \textbf{\indent else break}
    \EndIf
\EndFor 
\end{algorithmic}
\label{alg:successive_linear}
\end{algorithm}


\textbf{DD and Stability Guarantee}. We iteratively solve ~\eqref{eqn:unified_framework_linearized} in Algorithm~\ref{alg:successive_linear}. In Algorithm~\ref{alg:successive_linear}, ~\eqref{eqn:unified_framework_linearized} enforces DD~\eqref{subeqn:unified_dd} and solves an inner approximation of~\eqref{subeqn:unified_lyap}. When the algorithm is initialized at a feasible solution to~\eqref{eqn:unified_framework}, each iterate \reviewi{$z^k=(\alpha^k,P^k,F^k, X^k)$} will remain feasible for all $k >0$~\cite{dinh2011combining}, which ensures system stability and DD. 

\reviewii{We highlight that the contribution of our proposed method is not  algorithmic novelty. Instead, the contribution is to connect performance optimization over DDPF controllers to a bilinear optimization problem, which allows for optimal control synthesis approaches to DD and DDPF such as Algorithm~\ref{alg:successive_linear}.} 

\begin{lemma}[Convergence~\cite{dinh2011combining}]
     Algorithm~\ref{alg:successive_linear} converges to a KKT point $\reviewi{z}$ of~\eqref{eqn:unified_framework} if \reviewiii{1)}~\eqref{eqn:unified_framework} has non-empty domain $\mc{P}$, \reviewiii{2)} $J$ is convex in $\reviewi{z}$, \reviewiii{3)} $J(\reviewi{z}) > -\infty$  for all $\reviewi{z} \in \mc{P}$, and \reviewiii{4)} second order-sufficient condition holds for $\reviewi{z^k}$~\cite[Chp.3]{bonnans2013perturbation}. 
\end{lemma}
\begin{proof}
    Suppose a sequence of solution $\{\reviewi{z^k}\}_{k\geq0}$ generated by Algorithm~\ref{alg:successive_linear}, then its limit point belongs to the set of stationary points~\cite[Thm.4.3]{dinh2011combining}, and $J(\reviewi{z^{k+1}})-J(\reviewi{z^k})$ is negative and monotone decreasing~\cite[Lemma.4.2]{dinh2011combining}. Since $J(\reviewi{z})$ is bounded, $\lim_{k\rightarrow \infty}\norm{\reviewi{z^{k+1}}-\reviewi{z^k}}=0$ implies that $\{\reviewi{z^k}\}_{k\geq0}$ converges to a KKT point of~\eqref{eqn:unified_framework}.
\end{proof}

\textbf{Initial Condition}. Algorithm~\ref{alg:successive_linear} provides DD and stability guarantees when we initialize with a feasible solution to~\eqref{eqn:unified_framework}. This is equivalent to finding a DDPF controller, which can be solved via  the spectral abscissa method~\cite[Sec. 3.5.1]{Sarsilmaz2024}.

\section{Power Network Optimal Control}
We apply DD with internal stability to a four bus power network~\cite{Yi2019}. 
We empirically verify that our DD optimal control~\eqref{eqn:unified_framework} is DDPF over different system parameters. We also compare its output against SDP-$H_2$ control over convergence, actuation effort, $H_2$ norm and DD error~\eqref{eqn:dd_condition}.

\textbf{Power Network Modeling}.
\begin{figure}[]
\centering
\includegraphics[width=0.5\linewidth]{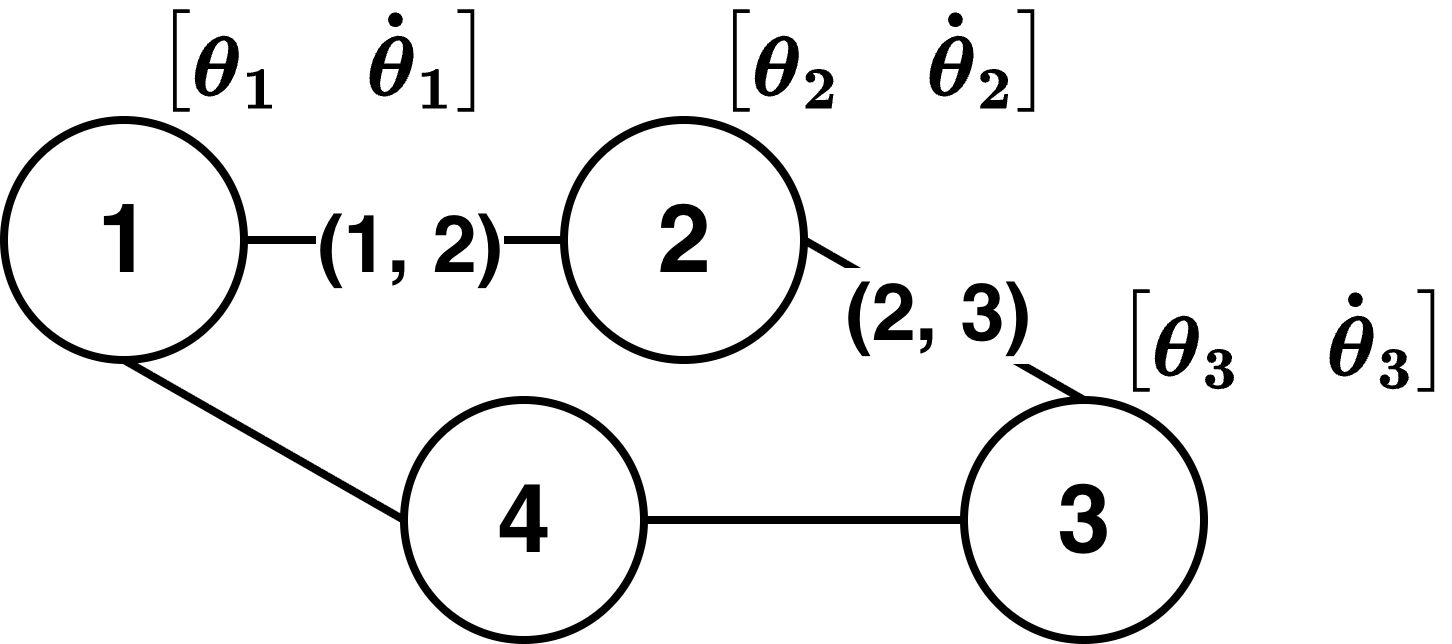}
\caption{Four-bus power network example. The (undirected) edges represents transmission lines between buses.}
\label{fig:powernetwork}
\end{figure}
We use a four bus power network from~\cite{Yi2019}, shown in Figure~\ref{fig:powernetwork} with bus 4 grounded as the infinite bus. Each node $i \in \{1,2,3\}$ represents a bus with two states: a voltage phase angle $\theta_i \in [0, 2\pi)$ and frequency $\dot{\theta}_i \in \reals_+$. Each bus also has a mechanical input torque $T_{M,i} \in \reals$\reviewiii{, which we use as the control input $u_i\in\mathbb{R}$}. The linearized generator dynamics around a stable operating point is given by~\cite{Yi2019,Tegling2015}
\begin{equation}\label{eqn:swing_equation}
   \textstyle M_i\ddot{\theta}_i+\beta_i\dot{\theta}_i = -\sum_{(i,j)\in\mathcal{E}}b_{ij}(\theta_i-\theta_j)+T_{m,i}. 
\end{equation}
For each bus $i$, the rotational inertia $M_i \in \reals_{+}$ has nominal value $10$ and the damping coefficient $\beta_i\in \reals_{+}$ has nominal value $\beta_i=10$. The transmission lines $(1,2),$ $ (2,3), (3,4), (4,1)$ respectively have susceptance $b_{ij}=\{0.386, 0.294, 0.596, 0.474\}$~\cite{Yi2019}.
Let $M=\diag{\{M_i\}}$, $D=\diag{\{\beta_i\}}$, and $L\in\mathbb{S}^3$ be the line susceptance Laplacian matrix. This power network has states $x=\begin{bmatrix}
        \theta^\top
        \dot{\theta}^\top 
    \end{bmatrix}^\top \in\reals^{6}$\reviewiii{, control $u \in \reals^{3}$,} and LTI matrices~\eqref{eqn:lti} given by
\begin{align*}
    A=\begin{bmatrix}
        0 & I\\
        -M^{-1}L & -M^{-1}D
    \end{bmatrix} \in\reals^{6\times6},\, 
    B=\begin{bmatrix}
        0\\
        M^{-1}
    \end{bmatrix} \in\reals^{6\times3}.
\end{align*}

\textbf{Disturbance Modeling}. We consider situations in which generator $3$ sustains stochastic mechanical torque deviations that alter the phases  of other generators and result in power loss in the entire network~\cite{Yi2019}. Stochastic mechanical torque deviations are additive noises to $\ddot{\theta}$ over time, such that  $E$ in~\eqref{eqn:lti} is $E=e_6$ and the external disturbance is modeled as 
\begin{equation}\label{eqn:noise_variance}
    d(t) \sim \mc{N}(0, \sigma), d(t) \in \reals, \quad \forall t \in \reals_+. 
\end{equation}
To ensure the nominal operation of buses 1 and 2 relative to the stable operating point, we use an output matrix given by $H=\begin{bmatrix}
    I_2 & 0_{4}
\end{bmatrix} \in \reals^{2\times 6}$ and outputs $ z(t) = Hx(t) \in \reals^2.$




\subsection{Case Study 1: DDPF verification}
We first verify that the optimal control problem~\eqref{eqn:unified_framework} finds a DD~\eqref{eqn:dd_condition} and internally stable~\eqref{eqn:internal_stab} controller. We solve~\eqref{eqn:unified_framework} with $J_{H_2} = \Tr(E^\top W_oE)$ and $\Pi=H^\top H+\epsilon I$ using Algorithm~\ref{alg:successive_linear}. 
\reviewiii{We solve a modified~\eqref{eqn:unified_framework} by replacing~\eqref{subeqn:unified_dd} with its rre form to satisfy  Theorem~\ref{thm:rcq}'s assumptions.}
We run $20$ Monte Carlo (MC) trials with each trial $k$ initialized with  randomized system parameters $M_i^k \sim \mc{N}(M_i, 1), \ \beta_i^k \sim \mc{N}(\beta_i, 1)$, initial states $x^k(0) \sim \mc{N}(\mathbf{1}_{6},I_{6\times6})$, and $d(t)\sim\mc{N}(0,1)$.
In Figure~\ref{fig: case1_study}, we plot the absolute difference in output $\vert{e(t)}\vert$~\eqref{eqn:dd_condition}, and the realized state trajectory $x(t)$ over $60$ seconds. We observe that $|e(t)|$ is numerically zero and the internal states converge. Therefore, controller is DD with internal stability.
\begin{figure}[]
\centering
\includegraphics[width=0.9\linewidth]{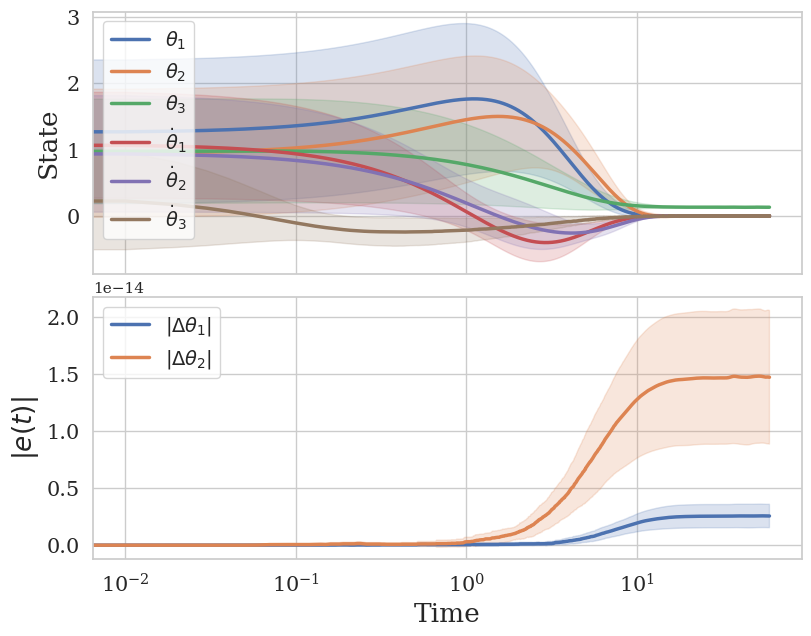}
\caption{The state evolution and output difference plot. }
\vspace{-8pt}
\label{fig: case1_study}
\end{figure}
\begin{table*}[]
\begin{minipage}[t]{0.52\textwidth}
\vspace{0pt}
\centering
\renewcommand{\arraystretch}{1.6}
\setlength{\tabcolsep}{1pt}
\begin{tabular}{|c|c|c|c|c|}
\hline
 & $f_{\alpha}$ & $f_{\mathrm{gain}}$ & $f_{H_2}$ & $f_{dd}$ \\
\hline
$F_{H_2}$ &
$\reviewi{-2e{-1} \!\pm\! 2e{-5}}$ &
$1e{5} \!\pm\! 1e{4}$ &
$4e{-14} \!\pm\! 1e{-14}$ &
$2e{-3} \!\pm\! 3e{-4}$ \\
\hline
$F_{dd-H_2}$ &
$\reviewi{-2e{-1} \!\pm\! 1e{-2}}$ &
$5e{1} \!\pm\! 8e{0}$ &
$\boldsymbol{8e{-25} \!\pm\! 9e{-25}}$ &
$\boldsymbol{7e{-13} \!\pm\! 5e{-13}}$ \\
\hline
$F_{\alpha}$ &
$\reviewi{\boldsymbol{-1e{0} \!\pm\! 1e{-1}}}$ &
$3e{1} \!\pm\! 5e{0}$ &
$3e{-18} \!\pm\! 3e{-18}$ &
$5e{-9} \!\pm\! 3e{-9}$ \\
\hline
$F_{gain}$ &
$\reviewi{-1e{-1} \!\pm\! 1e{-2}}$ &
$\boldsymbol{9e{-1} \!\pm\! 1e{-4}}$ &
$2e{-21} \!\pm\! 2e{-21}$ &
$5e{-12} \!\pm\! 2e{-12}$ \\
\hline
\end{tabular}
\vspace{3pt}
\caption{Controller performances comparison, shown as mean $\pm$ standard deviation. Each column's best mean value is bolded. }
\label{tab:controller_comparison}
\end{minipage}
\hfill
\begin{minipage}[t]{0.45\textwidth}
\vspace{0pt}
\raggedleft
\includegraphics[width=1\linewidth]{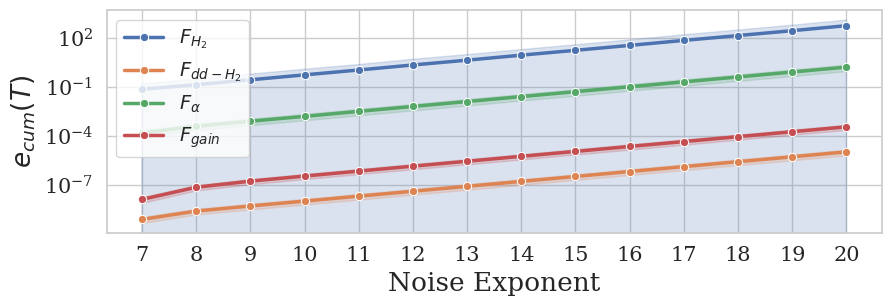}
\captionof{figure}{Cumulative time domain output error comparison over exponentially increasing noise variance. }
\label{fig:noise_scale}
\end{minipage}
\vspace{-8pt}
\end{table*}
\subsection{Case Study 2: DDPF vs SDP-$H_2$}\label{sec:simulation_case_2}
We show that \reviewiii{solving}~\eqref{eqn:unified_framework}, \reviewiii{ with~\eqref{subeqn:unified_dd} replaced by its rre form, }allows us to select from the set of DDPF feedback controls that optimize various performance criteria. To this end, we compare the performance of the following DDPF controllers under different criteria.
\begin{enumerate}
    \item $\boldsymbol{F_{H_2}}$: Solving~\eqref{eqn:h2_LMI} which minimizes $H_2$ norm via SDP.
    \item $\boldsymbol{F_{dd-H_2}}$: Solving~\eqref{eqn:unified_framework} for $J_{H_2} = \Tr(E^\top W_o E)$ and $\Pi = H^\top H+ \epsilon I$, which minimize $H_2$ norm.
    \item $\boldsymbol{F_\alpha}$: Solving~\eqref{eqn:unified_framework} for $J_\alpha = -\alpha$ and $\Pi = 2\alpha P$, which maximizes the system convergence rate.
    \item $\boldsymbol{F_{gain}}$: Solving~\eqref{eqn:unified_framework} for $J_{gain}=\norm{F}_2 $ and $\Pi = H^\top H+ \epsilon I$, which minimizes the control effort.
\end{enumerate}
Using the auxiliary optimization solutions $(\alpha, X, P)$ for each controller, we compare their performances via the metrics: $f_\alpha = -\alpha$, $ f_{H_2} = \Tr(E^\top W_o E)$, $f_{gain} = \norm{F}_2$, and $f_{dd} = \norm{VX - A_FV}_2$ ($f_{dd}=0$ certifies DD). For $\boldsymbol{F_{H_2}}$ which does not have auxiliary $\alpha, X,P$ solutions, we compute $\alpha$ as the maximum eigenvalue of $A_F$, $X = \argmin \norm{VX - (A+BF)V}_2 $, and $P$ by solving the observability Gramian. In Table~\ref{tab:controller_comparison}, we compare these metrics over $20$ randomized MC trials with randomized system parameters (see case study 1). For all four controllers, we also evaluate their  cumulative time domain error $e(T) = \int_{0}^T\norm{e(\tau)}_2d\tau$~\eqref{eqn:dd_condition} under increasing noise variance $d(t)\sim\mc{N}(0,2^l),l\in\{7, \dots, 20\}$, and compare $e_z(T)$ after $T=10s$.



\textbf{Discussions.}
We see that compared to $\boldsymbol{F_{H_2}}$, all controllers with subspace constraint~\eqref{eqn:dd_necessary_sufficient} achieved at least $1e4$ higher rates of $H_2$ norm minimization and at least $1e6$ higher rates of DD subspace condition satisfaction. We can also observe this in the cumulative error $\int_{0}^T\norm{e(\tau)}_2d\tau$ over increasingly larger noise magnitudes in Figure~\ref{fig:noise_scale}.
While \emph{all} controllers incur nonzero error over larger magnitudes, controllers $\boldsymbol{F_{dd-H_2}}$, $\boldsymbol{F_\alpha}$, and $\boldsymbol{F_{gain}}$  have significantly lower cumulative errors than $F_{H_2}$. 
\reviewii{Although the $H_2$ norm differences between SDP-based $H_2$ and DDPF controllers are extremely small, these differences have out-sized impacts in a larger power networks, where noise magnitudes accumulate from additional noise sources and over time. Consequently, the SDP-based $H_2$ controlled system output deviates from nominal system output more significantly, as shown in Figure 4.} 

In Table~\ref{tab:controller_comparison}, we observe $F_{dd-H_2}, F_\alpha, F_{gain}$ all outperform $H_2$ and achieve the lowest $f_{H_2}, f_\alpha, f_{gain}$, respectively. More critically, the control effort achieved is at least $1e4$ times lower than $\boldsymbol{F_{H_2}}$. This suggests that the optimization program~\eqref{eqn:unified_framework} produces more realizable $H_2$ control than classic SDP-based $H_2$ minimization. 
\bibliographystyle{IEEEtran}
\bibliography{reference}

\end{document}